\documentclass[12pt]{article}
\usepackage{amsfonts, amsmath, amsthm, enumerate}
\usepackage[pdftex]{graphics}
\usepackage[hang]{subfigure}
\usepackage{rotating}

\newcommand{\mb}{\mathbf}

\def\complaint#1{}
\def\withcomplaints{
\newcounter{mycomplaints}
\def\complaint##1{\refstepcounter{mycomplaints}%
\ifhmode%
\unskip%
{\dimen1=\baselineskip \divide\dimen1 by 2 %
\raise\dimen1\llap{\tiny -\themycomplaints-}}\fi%
\marginpar{\tiny [\themycomplaints]: ##1}}%
}

\withcomplaints 
\newtheorem{theorem}{Theorem}

\newtheorem{lemma}[theorem]{{Lemma}}
\newtheorem{prop}[theorem]{{Proposition}}
\newtheorem{question}{{Question}}

\def\eop{\hfill {\hfill $\Box$} \medskip}

\def\E{\mathbb E}
\def\R{\mathbb R}
\def\P{\mathbb P}

\def\p{{\bf p}}
\def\pn{{\bf p = (\bf p_1, \bf p_2, \dots, \bf p_n})}
\def\q{{\bf q}}

\def\v{{\bf v}}

\begin{document}

\title {Combining Globally Rigid Frameworks}

\author{R. Connelly \thanks{Research supported~ in~ part~ by~ NSF~ Grant~ No. DMS--0209595 (USA). 
\newline  e-mail:
connelly@math.cornell.edu} \\
Department of Mathematics, Cornell University\\
Ithaca, NY 14853, USA}
\maketitle 

\begin{abstract} Here it is shown how to combine two generically globally rigid bar frameworks in $d$-space to get another generically globally rigid framework.  The construction is to identify $d+1$ vertices from each of the frameworks and erase one of the edges that they have in common.  
\end{abstract}

\section{Introduction and definitions} \label{section:introduction}

Suppose that a finite configuration $\pn$ of labeled points in Euclidean $d$-dimensional space $\E^d$ is given, together with a corresponding graph $G$ whose vertices correspond to the points of $\p$.  Each edge of $G$, called a \emph{member}, is designated as a \emph{cable},  \emph{strut}, or  \emph{bar}.  All this data is denoted as $G(\p)$, and it is called a \emph{tensegrity}, or if all the members of $G$ are bars, $G(\p)$ is called a \emph{bar framework}.  

 We say the tensegrity $G(\p)$ \emph{dominates} the tensegrity $G(\q)$, and write $G(\q) \le G(\p)$, for two configurations $\q$ and $\p$, if 
\begin{eqnarray} 
|\p_i -\p_j| &\ge& |\q_i -\q_j| \quad\mbox{for $\{i,j\}$ a cable,}  \nonumber \\
|\p_i -\p_j| &\le& |\q_i -\q_j| \quad\mbox{for $\{i,j\}$ a strut and}  \label{tens-constraints}\\
|\p_i -\p_j| &=& |\q_i -\q_j| \quad\mbox{for $\{i,j\}$ a bar.} \nonumber 
\end{eqnarray}

This just means that going from $\p$ to $\q$, cables don't get longer, struts don't get shorter, and bars stay the same length.  For bar frameworks, we say $G(\p)$ \emph{is equivalent to} $G(\q)$ if $G(\q) \le G(\p)$ and $G(\p) \le G(\q)$ and we write $G(\q) \simeq G(\p)$. 

Two configurations $\p$ and $\q$ are \emph{congruent}, in $\E^d$, and we write $\p \cong \q$, if there is a $d$-by-$d$ orthogonal matrix $A$ and a vector $\mb{b} \in \E^d$ such that for all $i$, $\q_i = A\p_i + \mb{b}$.  Equivalently, $\p \cong \q$, if and only if $K(\p) \simeq K(\q)$, where $K$ is the complete graph on the vertices of $\p$ and $\q$.

A tensegrity $G(\p)$ is defined to be \emph{globally rigid} in $\E^d$ if for every tensegrity $G(\q)$ in $\E^d$ that is dominated by $G(\p)$, $\q$ is congruent to $\p$.  In other words,  $G(\q) \le G(\p)$ implies $\p \cong \q$.  So a bar framework $G(\p)$ is globally rigid in $\E^d$ if, for all configurations $\q$ in $\E^d$,  $G(\q) \simeq G(\p)$ implies $\p \cong \q$.  A tensegrity or a bar framework $G(\p)$ in $\E^d$ is \emph{universally globally rigid} or just \emph{universally rigid} if it is globally rigid in $\E^D \supset \E^d$ for all $D \ge d$.

A tensegrity $G(\p)$ is defined to be \emph{(locally) rigid} in $\E^d$ if there is an $\epsilon > 0$ and for $|\p-\q| < \epsilon$, any tensegrity $G(\q)$ in $\E^d$ that is dominated by $G(\p)$, $\q$ is congruent to $\p$.  In other words, $\q$ close enough to $\p$ and $G(\q) \le G(\p)$ implies $\p \cong \q$.  So a bar framework $G(\p)$ is locally rigid in $\E^d$ if there is an $\epsilon > 0$ such that  $|\p-\q| < \epsilon$, $\q$  in $\E^d$, and  $G(\q) \simeq G(\p)$ implies $\p \cong \q$.

A configuration $\p$ is said to be \emph{generic} if the coordinates of $\p$ in $\E^d$ are \emph{algebraically independent over the rational numbers}, which means that there is no non-zero polynomial with rational coordinates satisfied by the coordinates of $\p$.  This implies that no $d+2$ nodes lie in a hyperplane, for example, and a lot more.  Figure \ref{fig:general-examples} shows several examples of tensegrities and frameworks with and without the properties discussed here.

\begin{figure}[here]
    \begin{center}
        \includegraphics[width=.8\textwidth]{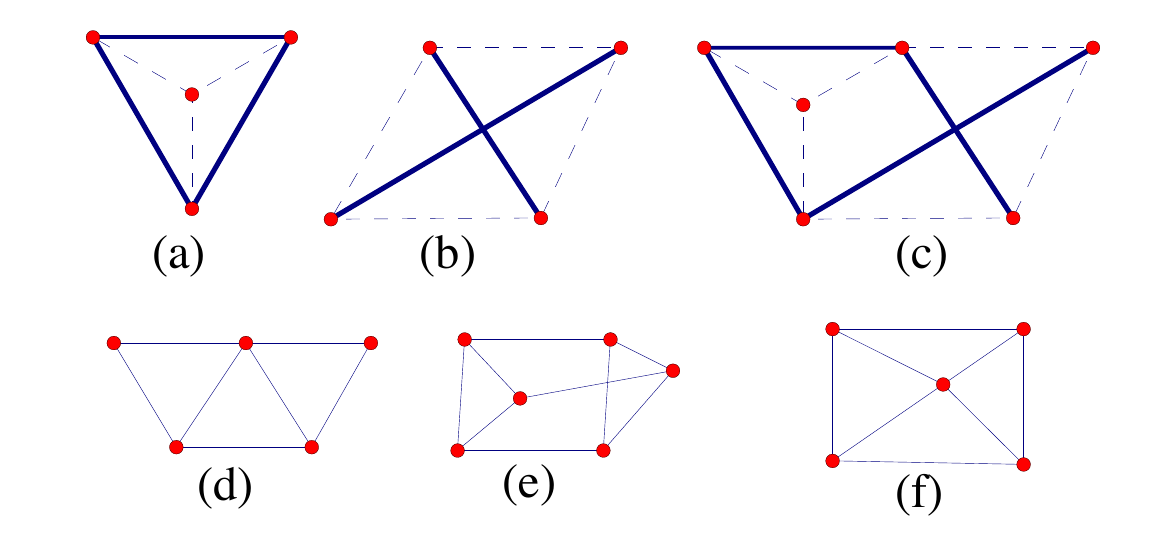}%
        \end{center}
    \caption{The top row shows tensegrities, where dashed lines are cables, solid lines struts.  The bottom row shows bar framworks.  All these examples are locally rigid, (a) and (b) are universally globally rigid, while (c), (d), and (e) are not globally rigid in the plane.  Example (f) is not globally rigid in the plane when the central point lies on the diagonal of the surrounding rectangle, but it is globally rigid in the plane when the configuration is generic.}
    \label{fig:general-examples}
    \end{figure}

\section{Basic previous results}

There has been a lot of work developing computationally feasible criteria for both local and global rigidity that involve purely combinatorial calculations for the graph $G$ and numerical criteria involving, additionally, the configuration $\p$.  A graph $G$ is called \emph{$m$-connected} if it takes the removal of, at least, $m$ vertices to disconnect $G$.  For example, in the plane $\E^2$ there is a popular algorithm, the pebble game, to compute, for a bar framework, whether $G(\p)$ is locally rigid when $\p$ is generic.  This algorithm is purely combinatorial, only depends on the graph $G$, and is polynomial in $n$, the number of vertices of $G$.  For information about this theory, see \cite{Graver-Servatius, Graver, Jacobs-Hendrickson, Whiteley-union}.  For all dimensions, determining whether a given bar framework $G(\p)$ is locally rigid at a generic configuration is also quite feasible, although it is not known to be feasible purely combinatorially.  For every bar framework $G(\p)$ in $\E^d$ with $n \ge d$ vertices, there is an associated $e$-by-$dn$ matrix $R(\p)$, the \emph{rigidity matrix}, such that $G(\p)$ is locally rigid in $\E^d$ if and only if the rank of $R(\p)$ is $ dv - d(d+1)/2$.

In order to understand some of the results about global rigidity it is helpful to look at the case of tensegrities, and in order to understand that it is helpful to understand stresses and stress matrices.  For any tensegrity $G(\p)$, a \emph{stress} $\omega = (\dots, \omega_{ij}, \dots)$ is a scalar $\omega_{ij}=\omega_{ji}$ associated to each member $\{i,j\}$ that connects vertex $i$ to vertex $j$ of $G$.  If vertex $i$ is not connected to vertex $j$, then $\omega_{ij}=0$.  We say the a stress $\omega$ for the tensegrity or framework $G(\p)$ is an \emph{equilibrium stress} if for all $j$, the following vector equation holds:

\begin{equation} \label{equilibrium}
\sum_i \omega_{ij}(\p_i-\p_j) = {\mb 0}.
\end{equation}

If  $G(\p)$ is a tensegrity, we say that $\omega$ is a \emph{proper stress} if $\omega_{ij} \ge 0$ for all cables $\{i,j\}$, and  $\omega_{ij} \le 0$ for all struts $\{i,j\}$.   If $\omega$ is a stress for $G(\p)$, and $G$ has $n$ vertices, form an $n$-by-$n$ symmetric matrix $\Omega$, called the \emph{stress matrix}, as follows:  Each off-diagonal entry of $\Omega$ is $-\omega_{ij}$, and the diagonal entries are such that the row and column sums are $0$.  Figure \ref{fig:square} shows a simple example of a tensegrity with a proper equilibrium stress indicated.

\begin{figure}[here]
    \begin{center}
        \includegraphics[width=.3\textwidth]{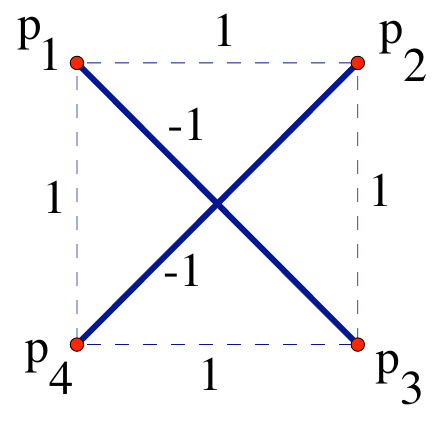}%
        \end{center}
    \caption{A square tensegrity with its diagonals, where a proper equilibrium stress is indicated.}
    \label{fig:square}
    \end{figure}

The stress matrix for this stress is

\begin{equation*} \label{stress-matrix}
\begin{pmatrix}
\,\,\,\,1& -1 & \,\,\,\,1 & -1 \\
- 1& \,\,\,\,1 & -1 & \,\,\,\,1 \\
 \,\,\,\,1& -1 & \,\,\,\,1 & -1 \\
 - 1& \,\,\,\,1 & -1 & \,\,\,\,1 
\end{pmatrix}.
\end{equation*}

In order to understand a fundamental theorem that implies universal global rigidity, we define the following concept.  Let $\v_1, \dots, \v_k$, be vectors in $\E^d$.  Regard these vectors as points in the real projective space $\R\P^{d-1}$ of lines through the origin in $\E^d$.   We say that $\v_1, \dots, \v_k$ \emph{lie on a conic at infinity} if, as points in $\R\P^{d-1}$ they lie on a conic (or quadric) hypersurface.  For example, in the plane $\E^2$, a conic at infinity consists of at most two points.   In $3$-space $\E^3$,  if we project the vectors into a plane, not through the origin, the conic is the usual notion of a conic, including the degenerate case of two lines.  The following is a fundamental result that has motivated a lot of the later results about global rigidity.  This can be found in \cite{Connelly-energy, Connelly-Whiteley}.  

\begin{theorem} \label{thm:fundamental}  Let $G(\p)$ be a tensegrity, where the affine span of $\p=(\p_1, \dots, \p_n)$ is all of $\E^d$, with a proper equilibrium stress $\omega$ and stress matrix $\Omega$.  Suppose further
\begin{enumerate}
\item \label{condition:positive} $\Omega$ is positive semi-definite.
\item \label{condition:rank} The rank of $\Omega$ is $n - d - 1$.
\item \label{condition:quadric} The set of vectors $\{\p_i - \p_j \mid \omega_{ij} \ne 0 \}$ do not lie on a conic at infinity. 
\end{enumerate}
Then $G(\p)$ is universally globally rigid.
\end{theorem}

In many cases, Condition \ref{condition:quadric}.) is easy to verify.  The difficulty usually lies with Condition \ref{condition:positive}.) and Condition \ref{condition:rank}.).  When the affine span of $\p$ is $d$-dimensional, then the rank of $\Omega$ is at most $n - d - 1$, because of the equilibrium conditions (\ref{equilibrium}).  When the three conditions of Theorem \ref{thm:fundamental} are satisfied we say that the tensegrity is \emph{super stable}.  

A partial converse to Theorem \ref{thm:fundamental} is the following result of S. Gortler,  A. D. Healey, and D. Thurston \cite{Gortler-Thurston}.
\begin{theorem} \label{thm:Gortler-Thurston}  Let $G(\p)$ be a universally globally rigid bar framework in $\E^d$, where $\p$ is generic and $G$ has at least $d+2$ vertices.  Then $G(\p)$ is super stable.  
\end{theorem}

So this means that under the conditions of Theorem \ref{thm:Gortler-Thurston}, there is an equilibrium stress such that the three conditions of Theorem \ref{thm:fundamental} hold.  So if the bars are converted to cables or struts to follow the sign of that stress, the bar constraints can be replaced by the much weaker inequality tensegrity constraints in (\ref{tens-constraints}).

Figure \ref{fig:general-examples}(a) and Figure \ref{fig:general-examples}(b) are super stable, while Figure \ref{fig:general-examples}(c) satisfies Condition \ref{condition:positive}.) and Condition \ref{condition:quadric}.), but not  Condition \ref{condition:rank}.) and, indeed, Figure \ref{fig:general-examples}(c) is not even globally rigid in the plane.

In order to understand Condition \ref{condition:rank}.) and use it, it helps to interpret the rank
condition on $\Omega$.  One very useful way to do this uses the following concept.  Suppose $\p$ is configuration with $n$ vertices in $\E^d$ with an equilibrium stress $\omega$.  We say the configuration $\p$ is \emph{universal with respect to $\omega$}  if,  when $\q$ is any other configuration on the same number of vertices such that $\omega$ is an equilibrium stress for $\q$, then the configuration $\q$ is an affine image of the configuration $\p$.  In other words, there is a $d$-by-$d$ matrix $A$ and a vector ${\mb v} \in \E^d$ such that $A\p_i + {\mb v} = \q_i$ for all $i=1, \dots, n$.  The following result in \cite{Connelly-energy} relates the notion of a universal configuration to the rank of the stress matrix.  We assume that the affine span of the configuration $\p$ is $d$-dimensional.

\begin{prop}  A non-zero equilibrium stress $\omega$ for a configuration $\p$ with $n$ vertices in $\E^d$ is universal if and only if the rank of the associated stress matrix $\Omega$ is $n-d-1$.
\end{prop}

A basis, including the vector of all one's, for the kernel of $\Omega$, $\ker(\Omega)$, can be used to construct a universal configuration as shown in \cite{Connelly-energy}.  For example, in $\E^d$ when the configuration $\p$ is universal with respect to the stress corresponding to $\Omega$, the $d$ vectors consisting of the $i$-th coordinates, for $i = 1, \dots, d$ and the vector of $n$ one's correspond to a basis for $\ker(\Omega)$.   When the emphasis is on a fixed configuration rather than a fixed equilibrium stress, we say $\Omega$ is of \emph{maximal rank} if its rank is $n-(d+1)$.

\section{Combining tensegrities}\label{sect:combining-tensegrities}

The stress matrix, since it is symmetric, can be regarded as a quadratic form on the space of all configurations $\p$, and we can add these quadratic forms as functions. (Technically, though it is the tensor product of $\Omega$ with the identity matrix $I^d$,  $\Omega \otimes I^d$, that corresponds to the quadratic form on the configurations $\p$.)  When we add positive semi-definite quadratic forms, the sum is positive semidefinite, and Condition \ref{condition:quadric}.) is also easy to verify in most cases.  It is also possible to check Condition \ref{condition:rank}.), when it is true.  For example, it is easy to see that Figure \ref{fig:general-examples}(c) is obtained by superimposing the rightmost strut in Figure \ref{fig:general-examples}(a) and the leftmost cable in \ref{fig:general-examples}(b).  If the stresses for (a) and (b) are adjusted by positive rescaling, the stress vanishes on the overlap.  But the rank of the stress matrix of the sum is $6 - (2+1) - 1 = 4$, one less than is needed for super stability.   Nevertheless, in any configuration dominated by Figure \ref{fig:general-examples}(c), in any dimension, is such that any pair of vertices both coming from either Figure \ref{fig:general-examples}(a) or both coming from Figure \ref{fig:general-examples}(b) have their distances preserved.  On the other hand if the overlap of two tensegrities consists of at least $d+1$ vertices, then the maximal rank Condition \ref{condition:rank}.) is preserved. 

 \begin{prop}\label{prop:combining-tensegrities} Suppose that $G_1(\p)$ and $G_2(\q)$ are two super stable tensegrities in  $\E^d$ with at least $d+1$ vertices in common, such that the $d+1$ vertices do not lie in a $(d-1)$-dimensional hyperplane, and such that one cable in $G_1$ overlaps with a strut in $G_2$.  Then the tensegrity  $G(\p \cup \q)$ obtained by superimposing their common vertices and members, but erasing the one common cable and strut, is also superstable.
  \end{prop}

It is understood that in $G$, if two other cables overlap, the resulting member in $G$ is a cable; if two struts overlap, the resulting member is a strut; and if another cable and strut overlap, the resulting member can be either a cable or strut or disappear, depending on the stresses of $G_1(\p)$ and $G_2(\q)$.  Figure \ref{fig:overlapping-tensegrities} shows an example of this.

\begin{figure}[here]
    \begin{center}
        \includegraphics[width=.8\textwidth]{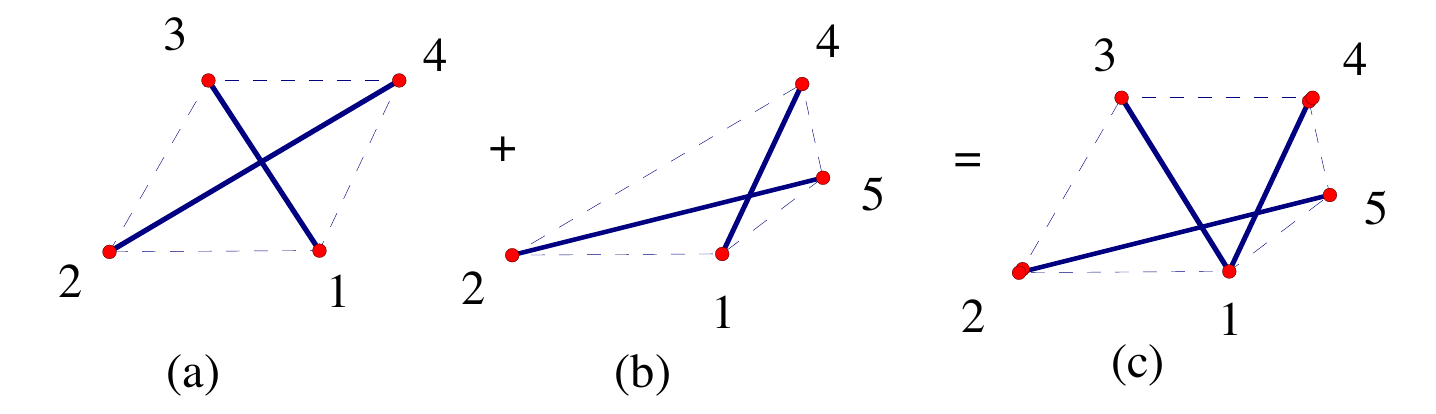}%
        \end{center}
    \caption{Figure (a) is combined with Figure (b) to get Figure (c) as with Proposition \ref{prop:combining-tensegrities}.  The stress in the $\{2,4\}$ strut in Figure (a) is scaled to cancel with the stress in the $\{2,4\}$ cable in Figure (b).  Note that the stress in the $\{1,4\}$ cable of in Figure (a) does not cancel with the stress in the $\{1,4\}$ strut in Figure (b).  The final stress in the $\{1,4\}$ member is negative and it is a strut in Figure (c) because of convexity of the five points and the equilibrium condition (\ref{equilibrium}).}
    \label{fig:overlapping-tensegrities}
    \end{figure}

The example of Figure \ref{fig:overlapping-tensegrities} is one case of a \emph{Cauchy polygon}, and Proposition \ref{prop:combining-tensegrities} is explained in more detail in \cite{Connelly-energy}.  Note that with this process, it is necessary to match a strut with a cable.

\section{Globally rigid generic bar frameworks}

For bar frameworks the story for global rigidity is different.  The starting point is to assume that the configuration $\p$ is generic, which has advantages and disadvantages. An advantage is that, in principle, generic global rigidity in $\E^d$ can be calculated with the help of some numerical calculation, but the downside is that the generic condition is hard to work with computationally.  For the case of local rigidity, the condition of being generic can be replaced by some polynomial conditions on the coordinates that are to be avoided.  For the case of global rigidity in $\E^d$, for $d \ge 3$, there are some polynomial conditions also that are to be avoided, but they seem to be intrinsically difficult to calculate.  The following basic result can serve as a starting point.  The ``if" part of the statement is due to \cite{Connelly-global}, and the ``only if" part is due to S. Gortler, A. D. Healy, and D. P. Thurston \cite{Gortler-Thurston}. 

\begin{theorem} \label{thm:generic-global}  Let $G(\p)$ be a bar framework at a generic configuration $\p$ in $\E^d$ with $n \ge d + 2$ vertices.  It is globally rigid in $\E^d$ if and only if there is a non-zero equilibrium stress whose stress matrix $\Omega$ has rank $n-d-1$.
\end{theorem}

The only globally rigid (generic) frameworks $G(\p)$ not covered in Theorem \ref{thm:generic-global} are when $G$ is the complete graph on less than $d+2$ vertices.  Note that Theorem \ref{thm:generic-global} essentially involves Condition \ref{condition:rank}.) of Theorem \ref{thm:fundamental}.  Condition \ref{condition:quadric}.) follows easily from the generic hypothesis and the equilibrium stress.

Note that a consequence of (the ``only if" part of and that the stress rank condition is a generic property) Theorem \ref{thm:generic-global} is that if $G(\p)$ is globally rigid at one generic configuration $\p$, then $G(\q)$ is globally rigid at all other generic configurations $\q$.  Furthermore, although generic configurations are hard to calculate concretely, it is enough to verify that for some configuration $\p$ the rank of the rigidity matrix $R(\p)$ is $dn - d(d+1)/2$, and the rank of a stress matrix is $n - d - 1$, as mentioned in \cite{Connelly-global, Connelly-Whiteley-coning}.

In dimension two the situation is even better, depends on the local rigidity properties of $G(\p)$ and depends on the combinatorics of $G$ only.  If $G(\p)$ is a bar framework in $\E^d$, is locally rigid and remains locally rigid after the removal of any bar, we say that $G(\p)$ is \emph{redundantly rigid} in $\E^d$. 

\begin{theorem}\label{thm:global-dim-2}
A bar framework $G(\p)$ with $\p$ generic is globally rigid in $\E^2$ if and only if $G(\p)$ is redundantly rigid and $3$-connected.
\end{theorem}

The ``only if" part of Theorem \ref{thm:global-dim-2} is due to B. Hendrickson in \cite{Hendrickson-unique}.  The ``if" part of Theorem \ref{thm:global-dim-2} is by A. Berg and T. Jordan; B. Jackson and T. Jordan; R. Connelly \cite{Berg-Jordan, Jackson-Jordan, Connelly-global}.  The pebble game of \cite{Jacobs-Hendrickson} provides an efficient purely combinatorial algorithm to compute generic redundant rigidity in the plane, and the computation of connectedness is known to have efficient polynomial time algorithms, so Theorem \ref{thm:global-dim-2} essentially provides a computationally effective method for computing generic global rigidity in the plane.

We say that a graph $G$ has the \emph{Hendrickson property} in $\E^d$ if $G$ is $(d+1)$-connected and $G(\p)$ is redundantly rigid in $\E^d$, when $\p$ is generic.  In \cite{Hendrickson-unique} B. Hendrickson shows the following:

\begin{theorem}\label{thm:Hendrickson}
If a bar framework $G(\p)$ with $\p$ generic is globally rigid in $\E^d$ then $G(\p)$ is redundantly rigid and $(d+1)$-connected.
\end{theorem}

Originally Hendrickson conjectured the converse of Theorem \ref{thm:Hendrickson} for $d\ge 3 $, but that is false since, in \cite{Connelly-K33}, it is shown that the complete bipartite graph $K(5,5)$ has the Hendrickson property in $\E^3$, but it is not globally rigid in $\E^3$, and there are other examples shown by S. Frank and J. Jiayang \cite{Frank-Jiayang}. In particular for $\E^5$, there are infinitely many examples, and similarly there are infinitely many examples for each $d \ge 5$.

\section{Combining generic globally rigid bar frameworks}

In $\E^d$ for $d \ge 3$, there is no known efficient deterministic combinatorial algorithm to compute generic global rigidity.  So it is reasonable to consider special combinatorial ways to create generically globally rigid bar frameworks from others especially in the spirit of Section \ref{sect:combining-tensegrities}.

One very natural way to combine two frameworks is to assume some overlap of the vertices and remove some of the members joining the common vertices.  If some members belong to one side, but not the other, the following natural result by K. Ratmanski \cite{Ratmanski} is useful.

\begin{theorem}\label{thm:combining-frameworks-Ratmanski} Suppose that $G_1(\p)$ and $G_2(\q)$ are globally rigid bar frameworks in  $\E^d$ with $d+1$ vertices (or more) in common such that $\p \cup \q$ is generic.  Let $G$ be the graph obtained by taking the union of their vertices and members, but  deleting those members from $G_2$ not in $G_1$.  Then the bar framework  $G(\p \cup \q)$ is also globally rigid in $\E^d$.
  \end{theorem}
  \begin{proof} This follows directly from the statement of global rigidity.  Suppose that the framework $G(\p \cup \q)$ is equivalent to $G(\hat{\p} \cup \hat{\q})$ in $\E^d$.  Since $G_1(\p)$ is globally rigid, the configurations $\p$ and $\hat{\p}$ are congruent.  So all the lengths of members in $\p \cap \q$ are preserved.  So  $\q$ and $\hat{\q}$ are congruent since  $G_2(\q)$ is globally rigid.  Since $\p \cup \q$ are generic and there are $d+1$ vertices in common, $\hat{\p} \cup \hat{\q}$ is congruent to $\p \cup \q$. \eop 
  \end{proof}

In order to treat the case when we delete a common member, first consider the following. We need an elementary Lemma from linear algebra.  
 
 \begin{lemma} \label{lem:sum}  Suppose that $\Omega_1$ and $\Omega_2$  are two $n$-by-$n$ symmetric matrices, such that the dimension of  $\ker \Omega_1 \cap \ker \Omega_2$ is $k$, and the rank of $ \mathrm{rank}\{ \Omega_i\} = r_i, \,\, i=1, 2$, where $r_1 + r_2 = n-k$.
Then 
 \begin{equation}\label{eqn:span} \mathrm{rank}  \{ t \Omega_1 +  (1-t) \Omega_2\} = n-k , \,\,  t \ne \pm1.
 \end{equation}
 \end{lemma} 
 \begin{proof} Since $\Omega_1$ and $\Omega_2$ are symmetric, and  $\ker \Omega_1 \cap \ker \Omega_2$ is an invariant subspace of both $\Omega_1$ and $\Omega_2$, we can restrict to the orthogonal complement of $\ker \Omega_1 \cap \ker \Omega_2$.  So we may assume, without loss of generality, that $k=0$.
 
  Again, since $\Omega_1$ and $\Omega_2$ are both symmetric, the orthogonal complements $(\ker \Omega_1)^{\perp}$ and $(\ker \Omega_2)^{\perp}$  are the images  $\mbox{Im}\, \Omega_1 = (\ker \Omega_1)^{\perp}$, $\mbox{Im}\, \Omega_2 = (\ker \Omega_2)^{\perp}$, respectively. Because  $[(\ker \Omega_1)^{\perp} + (\ker \Omega_2)^{\perp}]^{\perp} = \ker \Omega_1 \cap \ker \Omega_2 =  \{\mb{0}\}$, then $(\ker \Omega_1)^{\perp} + (\ker \Omega_2)^{\perp} = \R^n$.  In other words, the combined images of $\Omega_1$ and $\Omega_2$ span.   Since $r_1+r_2 = n$  these spaces are complementary in $\R^n$.  Thus for all $t \ne 0, 1$, $ t \Omega_1 +  (1-t) \Omega_2$ is non-singular.  \eop  
  \end{proof}

  
 We next apply this to stress matrices.
 
 \begin{lemma}\label{lem:overlap}Suppose that $G_1(\p)$ and $G_2(\q)$ are two bar frameworks in $\R^d$, with $n_1$ and $n_2$ vertices, respectively, that share exactly $d+1$ vertices not lying in a $(d-1)$-dimensional hyperplane, and with corresponding stress matrices $\Omega_1$ and $\Omega_2$.  Extend $\Omega_1$ to  $\tilde{\Omega}_1$ to include the vertices $\q$ of $G_2$ not in $G_1$, but with $0$ stress on all the extra pairs of vertices.  Similarly extend $\Omega_2$ to  $\tilde{\Omega}_2$.  If each $\Omega_i$ has maximal rank $n_i - (d+1)$, then for all values of $t \ne 0, 1$, $t \tilde{\Omega}_1+ (1-t) \tilde{\Omega}_2$ has maximal rank $n_1 + n_2 - 2(d+1)$.  
 \end{lemma}
 
\begin{proof} By the maximal rank condition, $\dim (\ker \Omega_1)= \dim (\ker \Omega_2) = d+1$, and $\dim (\ker \tilde{\Omega}_1) = d+1 + n_2-(d+1)$ while $\dim (\ker \tilde{\Omega}_2) = d+1 + n_1-(d+1)$.   Each of these kernels corresponds to a universal configuration that has the vertices of $\p$ and $\q$, for $\tilde{\Omega}_1$ and $\tilde{\Omega}_2$, respectively, such that they lie in a $d$-dimensional affine linear space, while the extra vertices each correspond to a higher dimensional configuration.

 The union of the vertices of $\p$ and $\q$, $\p \cup \q$, is a configuration that satisfies the equilibrium equations of the stresses corresponding to both $\tilde{\Omega}_1$ and $\tilde{\Omega}_2$.  If $\bar{\p} \cup \bar{\q}$ is another configuration that satisfies the equilibrium equations of the stresses corresponding to both $\tilde{\Omega}_1$ and $\tilde{\Omega}_2$, then $\bar{\p}$ is an affine image of $\p$ and $\bar{\q}$ is an affine image of $\q$, since $\p$ and $\q$ are universal with respect to the stresses corresponding to $\Omega_1$ and $\Omega_2$, respectively.  Thus, by extending the correspondence between $\p \cap \q$ and  $\bar{\p} \cap \bar{\q}$, we get an affine map from $\p \cup \q$ to  $\bar{\p} \cup \bar{\q}$.  Thus $\p \cup \q$ corresponds to a basis for the intersection of the kernels of $\tilde{\Omega}_1$ and $\tilde{\Omega}_2$.  The affine span of $\p$ and $\q$ are both $d$-dimensional with $d+1$ affine independent points in the intersection, so  $\p \cup \q$ has a $d$-dimensional affine span.   In other words, $\ker\tilde{\Omega}_1 \cap \ker \tilde{\Omega}_2$ has dimension $d+1$.  Then Lemma \ref{lem:sum} implies the conclusion with $k= d+ 1$, and $ \mathrm{rank}\{ \Omega_i\} = r_i, \,\, i=1, 2$, since $n=n_1+n_2 - (d+1)$, and $r_1 +r_2 = n_1 -(d+1) + n_2 -(d+1)  = n - (d+1)$.  \eop
\end{proof}

\section{The main theorem}

\begin{theorem}\label{thm:combining-frameworks} Suppose that $G_1(\p)$ and $G_2(\q)$ are globally rigid bar frameworks in  $\E^d$, with $\p \cup \q$ generic, exactly $d+1$ vertices in common, each with at least $d+2$ vertices, and a bar $\{i,j\}$ in  $G_1$ and $G_2$.  Then the bar framework  $G(\p \cup \q)$ obtained by superimposing their common vertices and bars, but erasing the bar $\{i,j\}$, is also globally
 rigid in $\E^d$.
  \end{theorem}
\begin{proof} By Theorem \ref{thm:generic-global} there are non-zero stress matrices $\Omega_1$ for $G_1(\p)$, and $\Omega_2$  for $G_2(\q)$ such that $\mathrm{rank}\{ \Omega_1\} = n_1- (d+1) \ge 1$, and  $\mathrm{rank}\{ \Omega_2\} = n_2- (d+1) \ge 1$ where $n_1$ is the number of vertices of $G_1$, and  $n_2$ is the number of vertices of $G_2$.  Then Lemma \ref{lem:overlap} implies that for  $t \ne 0, 1$, $t \tilde{\Omega}_1+ (1-t) \tilde{\Omega}_2$ has maximal rank $n_1 + n_2 - 2(d+1)$.  Let $\omega_{ij}(1)$ and $\omega_{ij}(2)$ be the stresses corresponding to $\Omega_1$ and $\Omega_2$, respectively for the bar $\{i,j\}$.  If either $\omega_{ij}(1)=0$ or $\omega_{ij}(2)=0$, Theorem \ref{thm:combining-frameworks-Ratmanski} implies that  $G(\p \cup \q)$ is globally rigid in $\E^d$.  Otherwise by rescaling $\Omega_1$ and $\Omega_2$, if necessary,  we can assume that  $\omega_{ij}(1)= 1$ and  $\omega_{ij}(2) = -1$.  Then Lemma \ref{lem:overlap}, with $t=1/2$, implies that there is a stress matrix, with maximal rank, such that the stress on $\{i,j\}$ is $0$, which allows us to remove it.  Then Theorem \ref{thm:generic-global} applies again to show that the resulting framework with $\{i,j\}$ deleted is globally rigid in $\E^d$.
\eop
\end{proof}

Figure \ref{fig:overlapping-tensegrities} is a typical example in the plane of Theorem \ref{thm:combining-frameworks}, but where the members are interpreted as bars.

It would be interesting to consider the case when there are more than $d+1$ vertices in common, but the method here does not seem to apply directly, since there may be linear combinations of the two stresses that are of lower, and we can't zero out a stress on a given member keeping the maximal rank condition.  For example, when there are $d+2$ vertices in common in $\E^d$, and there is a vertex in the intersection of degree $d+1$, no maximal rank  linear combination of the two stresses can zero out the stress on one bar .  (This is an observation of Tibor Jord\'an.)

 \begin{question} \label{quest:hendricson} Suppose that we combine two graphs that have the Hendrickson property as in Theorem \ref{thm:combining-frameworks}.   Does the resulting framework have the Hendrickson property.
 \end{question}
 
 It seems that the connectivity property holds for the resulting framework.  I don't know about the redundant rigidity property for $d \ge3$.  For $d=2$, the statement of Question \ref{quest:hendricson} is true because the Hendrickson property and generic global rigidity are equivalent by \cite{Jackson-Jordan}.  But, for $d=2$, the redundant rigidity condition holds by itself, without the need of the connectivity condition by a recent result of Bill Jackson and Tibor Jord\'an.
 
 The proof of Lemma \ref{lem:sum} and Lemma \ref{lem:overlap} was inspired by a draft Lemma in \cite{Connelly-Jordan-Whiteley} that was incorrect.  The author is very grateful for Dylan Thurston and Tibor Jord\'an for pointing out a previous incorrect statement (and proof of course) of Lemma \ref{lem:sum}.   The author also thanks Igor Gorbovickis for several useful comments and corrections.  For other related results, see \cite{Jacobs, Sun-Ye, Gortler-Thurston2, Jackson-linked, Connelly-Whiteley-coning, Jordan-Szabadka, Jackson-linked, Alfakih-Ye}.

\bibliographystyle{plain}
\bibliography{NSF-10,framework}


\end{document}